\documentclass[11pt]{amsart}
\usepackage{amsmath,amssymb,latexsym,soul,cite,amsthm,color,enumitem,graphicx,tikz, mathtools,microtype}
\usepackage[colorlinks=true,urlcolor=airforceblue,citecolor=airforceblue,linkcolor=airforceblue,linktocpage,pdfpagelabels,bookmarksnumbered,bookmarksopen]{hyperref}
\definecolor{airforceblue}{rgb}{0.36, 0.54, 0.66}
\usepackage[english]{babel}
\usepackage[left=2.9cm,right=2.9cm,top=2.6cm,bottom=2.6cm]{geometry}

\numberwithin{equation}{section}

\newtheorem{theorem}{Theorem}[section]
\theoremstyle{plain}
\newtheorem{lemma}[theorem]{Lemma}
\theoremstyle{plain}

\theoremstyle{plain}

\theoremstyle{definition}

\newcommand{\N}{{\mathbb N}}

\newcommand{\R}{{\mathbb R}}
\newcommand{\eps}{\varepsilon}
\newcommand{\beq}{\begin{equation}}
\newcommand{\eeq}{\end{equation}}
\renewcommand{\le}{\leqslant}
\renewcommand{\ge}{\geqslant}

\newcommand{\W}{\mathcal{W}}
\newcommand{\C}{\mathcal{C}}
\newcommand{\fpl}{(-\Delta)_p^s\,}
\newcommand{\ds}{{{\rm d}}_\Omega^s}
\newcommand{\p}{p^*_s}

\def\XXint#1#2#3{{\setbox0=\hbox{$#1{#2#3}{\int}$ }
\vcenter{\hbox{$#2#3$ }}\kern-.6\wd0}}

\newenvironment{enumroman}{\begin{enumerate}

}{\end{enumerate}}

\title[Sobolev vs.\ H\"older minimizers for the fractional $p$-Laplacian]{Sobolev versus H\"older minimizers for the degenerate\\ fractional ${\bf p}$\,-\,Laplacian}

\author[A.\ Iannizzotto, S.\ Mosconi, M.\ Squassina]{Antonio Iannizzotto, Sunra Mosconi, and Marco Squassina}

\address[A.\ Iannizzotto]{Department of Mathematics and Computer Science
\newline\indent
Universit\`a degli Studi di Cagliari
\newline\indent
Viale L.\ Merello 92, 09123 Cagliari, Italy}
\email{antonio.iannizzotto@unica.it}

\address[S.\ Mosconi]{Dipartimento di Matematica e Informatica
\newline\indent
Universit\`a degli Studi di Catania
\newline\indent
Viale A.\ Doria 6, 95125 Catania, Italy}
\email{mosconi@dmi.unict.it}

\address[M.\ Squassina]{Dipartimento di Matematica e Fisica
\newline\indent
Universit\`a Cattolica del Sacro Cuore
\newline\indent
Via dei Musei 41, 25121 Brescia, Italy}
\email{marco.squassina@unicatt.it}


\begin{document}

\begin{abstract}
We consider a nonlinear pseudo-differential equation driven by the fractional $p$-Laplacian $\fpl$ with $s\in(0,1)$ and $p\ge 2$ (degenerate case), under Dirichlet type conditions in a smooth domain $\Omega$. We prove that local minimizers of the associated energy functional in the fractional Sobolev space $W^{s,p}_0(\Omega)$ and in the weighted H\"older space $C^0_s(\overline\Omega)$, respectively, do coincide.
\end{abstract}

\maketitle

\section{Introduction and main result}\label{sec1}

\noindent
The 'Sobolev versus H\"older minimizers problem' is a classical one in nonlinear analysis, arising from the application of variational methods to boundary value problems of the following general type:
\beq\label{bvp}
\begin{cases}
Lu=f(x,u) & \text{in $\Omega$} \\
u\in\W.
\end{cases}
\eeq
Here $\Omega\subset\R^N$ is a (generally bounded and smooth) domain, $\W$ is a Sobolev-type function space defined on $\Omega$, incorporating some boundary condition, $L:\W\to\W^*$ is a (linear or nonlinear) elliptic differential operator in divergence form, and the reaction $f:\Omega\times\R\to\R$ is a Carath\'eodory mapping satisfying suitable growth conditions. In many relevant cases, weak solutions of problem \eqref{bvp} coincide with critical points of an energy functional $J\in C^1(\W)$ of the form
\beq\label{j}
J(u)=\frac{\|u\|_\W^p}{p}-\int_\Omega F(x,u)\,dx,
\eeq
where $p>1$ and $F(x,\cdot)$ denotes a primitive of $f(x,\cdot)$. Among critical points, local minimizers play a special r\^ole, as they are the starting point to apply mountain pass or minimax schemes, as well as Morse-theoretic arguments aimed at multiplicity results. 
\vskip2pt
\noindent 
In order to localize solutions of \eqref{bvp},  truncations of the reaction are often employed: as a typical example, when $f(\cdot,0)=0$, positive solutions are detected by using a modified functional $J_+$ defined as in \eqref{j}, but with $f$ replaced by $f_+(x,t)=f(x,t^+)$ (here $t^+$ denotes the positive part of $t\in\R$). Alternatively, one may truncate $f(x,\cdot)$ outside a sub-solution and a super-solution of \eqref{bvp} in order to 'trap' solutions within a given functional interval. Among many examples, we refer the reader to the classical works \cite{ABC,LL,W}. In all these cases natural constraints are employed, so that weak comparison arguments ensure that critical points of the truncated functionals  are  critical points of $J$ as well. A drawback of the truncation method is that the topological nature of such critical points is {\em a priori} lost in the process, along with the valuable information that can be derived from it: in particular, it is not {\em a-priori} ensured that local minimizers of the truncated functionals are minimizers of $J$. For instance, this is the case for the truncation at $0$, as $J$ and $J_+$ only agree on the positive cone $\W_+$, which in general has an empty interior.
\vskip2pt
\noindent
In \cite{BN}, Brezis and Nirenberg proposed an answer to this issue for the Dirichlet problem with $L=-\Delta$ and $\W=H^1_0(\Omega)$, by proving that local minimizers of $J$ in $H^1_0(\Omega)$ coincide with those in the space $\C=C^1(\overline\Omega)$ (whose positive cone has a nonempty interior). Such result relies on classical elliptic regularity theory, as well as on the linearity of the operator (coincidence is first proved for $0$, and then extended to any minimizer $u$ by translation). A key point is that minimizers of $J$ on closed balls in $H^1_0(\Omega)$ solve a problem of the form \eqref{bvp}, involving a Lagrange's multiplier as well.
\vskip2pt
\noindent
When nonlinear operators are considered, the question becomes more involved. In \cite{GPM}, Garc\`{\i}a Azorero, Peral Alonso and Manfredi extended the coincidence result to the $p$-Laplacian operator $L=-\Delta_p$, with $\W=W^{1,p}_0(\Omega)$ ($p>1$) and again $\C=C^1(\overline\Omega)$. In this case, to deal with non-zero minimizers, a simple translation would not work, but the authors overcame such difficulty by employing the nonlinear regularity theory of Lieberman \cite{L} for a more general operator than $-\Delta_p$. Since then, 'Sobolev versus H\"older' results were proved for a number of \eqref{bvp}-type problems involving several operators (both linear and nonlinear) and boundary conditions (see for instance \cite{F1,GP,IP}). In particular, we mention the approach of Brock, Iturriaga and Ubilla \cite{BIU}, where constrained minimization of $J$ is performed on balls in $L^q(\Omega)$ ($q<p^*$), so that the modified equation is still of $p$-Laplacian type, though involving an additional power term (see also \cite{FGU}).
\vskip2pt
\noindent
When it comes to nonlocal operators of fractional order, problem \eqref{bvp} is naturally set in a fractional Sobolev space $\W=H^s_0(\Omega)$, but $C^1$-regularity up to the boundary is not to be expected any more. For instance, given $s\in (0, 1)$, the function $(1-|x|^{2})_{+}^{s}$ solves 
\[\begin{cases}
(-\Delta)^{s}u=1 & \text{in $B_{1}$}\\
u= 0 & \text{in $\R^{N}\setminus B_{1}$,}
\end{cases}\]
but clearly $|\nabla u|$  blows up near the boundary. The issue, moreover, does not only involve the boundary behaviour: for $f\in L^{\infty}(\R^{N})$, the optimal interior regularity for solutions of $(-\Delta)^{s}u=f$  is $C^{2s}$ when $s\neq 1/2$, so that when $s<1/2$ we cannot expect even Lipschitz continuity {\em in the interior} (see \cite{RS}). In the fractional framework, the natural function space to work with is constructed through a {\em weighted} H\"older regularity condition, namely assuming that $u/\ds$ admits a continuous extension to $\overline\Omega$, where $\ds(x)={\rm dist}(x,\R^{N}\setminus \Omega)^s$. We denote by $\C=C^0_s(\overline\Omega)$ the space of such $u$'s, equipped with the norm $\|u\|_{C^{0}_{s}}=\sup_{\Omega}|u|/\ds$, while we set $L=(-\Delta)^s$ (the fractional Laplacian of order $s\in(0,1)$) and $\W=H^s_0(\Omega)$. In \cite{IMS}, equivalence of minimizers for $J$ in $\W$ and in $\C$ was proved following the approach of \cite{BN} (see also \cite{BCSS,F} for similar results). 
\vskip2pt
\noindent
In this work we propose a 'Sobolev versus H\"older' result for a nonlinear nonlocal equation driven by the degenerate fractional $p$-Laplacian, namely the nonlinear extension of $(-\Delta)^s$. For any $s\in(0,1)$, $p\ge 1$, $N>ps$ we define the Gagliardo (semi-)norm of a measurable function $u:\R^N\to\R$ as
\[\|u\|_{s,p}^p = \iint_{\R^{2N}}|u(x)-u(y)|^p\,d\mu,\]
where we use the abbreviated notation
\[d\mu = \frac{dx\,dy}{|x-y|^{N+ps}}.\]
Further, let $\Omega\subset\R^N$ be a bounded domain with a $C^{1,1}$-boundary and set
\[W^{s,p}_0(\Omega) = \big\{u\in L^p(\R^N):\,\|u\|_{s,p}<\infty,\,u=0\,\text{a.e.\ in $\R^N\setminus\Omega$}\big\}.\]
The space $W^{s,p}_0(\Omega)$, endowed with the norm $\|\cdot\|_{s,p}$, is a separable, uniformly convex Banach space with dual denoted by $W^{-s,p'}(\Omega)$. The embedding $W^{s,p}_0(\Omega)\hookrightarrow L^q(\Omega)$ is continuous for all $q\in[1,\p]$ and compact for all $q\in[1,\p)$, where
\[\p = \frac{Np}{N-ps}\]
denotes the fractional Sobolev exponent. We define the fractional $p$-Laplacian as an operator $\fpl:W^{s, p}_0(\Omega)\to W^{-s, p'}(\Omega)$ given by
\[\langle\fpl u,\varphi\rangle = \iint_{\R^{2N}}(u(x)-u(y))^{p-1}(\varphi(x)-\varphi(y))\,d\mu,\]
i.e., $\fpl$ is the Fr\'echet derivative of $u\mapsto \|u\|_{s,p}^{p}/p$ (see \cite{ILPS} for details). Note that $\fpl$ is both nonlinear and nonlocal. We will consider the following Dirichlet problem:
\beq\label{dir}
\begin{cases}
\fpl u=f(x,u) & \text{in $\Omega$} \\
u=0 & \text{in $\R^{N}\setminus \Omega$,}
\end{cases}
\eeq
$f:\Omega\times\R\to\R$ is a Carath\'eodory mapping obeying the following at most critical growth condition for a.e.\ $x\in\Omega$ and all $t\in\R$:
\beq\label{gro}
|f(x,t)|\le C_0\, (1+|t|^{\p-1}) \qquad (C_0>0).
\eeq
According to the general formula \eqref{j}, the energy functional for problem \eqref{dir} is $J\in C^1(W^{s,p}_0(\Omega))$, defined by
\[J(u) = \frac{\|u\|_{s,p}^p}{p}-\int_\Omega F(x,u)\,dx,\]
with
\[F(x,t) = \int_0^t f(x,\tau)\,d\tau.\]
In the present case, the r\^ole of the space $\C$ is played by the weighted H\"older space
\[C^0_s(\overline\Omega) = \Big\{u\in C^0(\overline\Omega):\,\frac{u}{\ds}\in C^0(\overline\Omega)\Big\},\]
endowed with the norm
\[\|u\|_{C^{0}_{s}}=\Big\|\frac{u}{\ds}\Big\|_\infty.\]
Our main result is the following, proving coincidence of Sobolev and H\"older minimizers of $J$:

\begin{theorem}\label{svh}
Let $p\ge 2$, $s\in (0, 1)$, $N>ps$, $\Omega\subseteq \R^{N}$ be a bounded domain with a $C^{1,1}$-boundary, $f:\Omega\times\R\to\R$ be a Carath\'eodory mapping satisfying \eqref{gro}. Then, for any $u_0\in W^{s, p}_0(\Omega)$, the following are equivalent:
\begin{enumroman}
\item\label{svh1} there exists $\rho>0$ such that $J(u_0+v)\ge J(u_0)$ for all $v\in W^{s, p}_0(\Omega)$, $\|v\|_{s,p}\le\rho$;
\item\label{svh2} there exists $\sigma>0$ such that $J(u_0+v)\ge J(u_0)$ for all $v\in W^{s, p}_0(\Omega)\cap C^{0}_{s}(\overline\Omega)$, $\|v\|_{C^{0}_{s}}\le\sigma$.
\end{enumroman}
\end{theorem}

\noindent
We make some comments on Theorem \ref{svh}:
\begin{itemize}
\item[$(a)$] {\em Choice of the space}. There are many reasons why $C^{0}_{s}(\overline\Omega)$ is a natural choice where to settle this kind of result. Mainly, such choice is dictated by the results of \cite{IMS2} (see also \cite{IMS1,IMS2}), where an {\em a priori} bound for solutions of \eqref{dir} in the space $C^{\alpha}_{s}(\overline\Omega)=\{u\in C^{\alpha}(\overline\Omega):u/\ds\in C^{\alpha}(\overline\Omega)\}$ for some $\alpha>0$, which compactly embeds into $C^{0}_{s}(\overline\Omega)$ (see Section \ref{sec2} for details). Plus, in \cite{DQ} (see also \cite{J}) the following version of Hopf's lemma was proved: any solution $u\ge 0$ of \eqref{dir} with non-negative right hand side either vanishes identically, or $u/\ds\ge c$ for some $c>0$. In particular, these signed solutions belongs to the interior of the non-negative cone in $C^{0}_{s}(\overline\Omega)$.
\item[$(b)$] {\em Method of proof.} Our strategy is more in the spirit of \cite{BIU} rather than of \cite{GPM}, with constrained minimization on $L^{\p}(\Omega)$-balls in order to deal with possibly critical problems, and employs as well a special monotonicity property of $\fpl$ (see Section \ref{sec3} for the detailed proof). Notice that in references such as \cite{BN,GPM,BIU} only the implication \ref{svh2} $\Rightarrow$ \ref{svh1} is considered, as the other one is trivial due to  $C^1(\overline\Omega)\hookrightarrow W^{1,p}_0(\Omega)$. A typical feature of the nonlocal framework is that $C^{0}_{s}(\overline\Omega)$ is not included in $ W^{s, p}_0(\Omega)$, so we have to prove both implications.
\item[$(c)$] {\em Applications.} The semi-linear case $p=2$ of Theorem \ref{svh} proved in \cite{IMS} has already been applied in a number of settings, see e.g. \cite{DI,FP,U} or the survey \cite{MS}. We believe that Theorem \ref{svh} will prove to be equally useful in the quasi-linear setting. Already in \cite[Theorem 5.3]{ILPS}, a multiplicity result for problem \eqref{dir} was proved under the conjecture that a version of Theorem \ref{svh} holds: such result is now fully achieved. In Section \ref{sec4}, we will briefly describe an illustrative application.
\item[$(d)$] {\em The singular case.} We remark that our result is only proved in the degenerate (or superquadratic) case $p\ge 2$. This is due to the fact that the $C^\alpha_s(\overline\Omega)$-regularity mentioned above has so far only been proved in this setting. For the singular case $p\in(1,2)$, the boundary regularity issue is therefore still open, nevertheless the corresponding case of Theorem \ref{svh} could be easily proven with only slight modifications, using Lemma \ref{smon} toghether with the putative singular counterpart of \cite{IMS3}.
\end{itemize}
\vskip4pt
\noindent
{\bf Notation.} Throughout the paper we will use the short notation $a^q=|a|^{q-1}a$ for all $a\in\R$, $q\ge 1$. We will denote $\|\cdot\|_q$ the usual norm of $L^q(\Omega)$ for all $q\in[1,\infty]$. Finally, $C$ will denote several positive constants, only depending on the data of the problem.

\section{Preliminaries}\label{sec2}

\noindent
In this section we introduce some technical results which will be used in the proof of our main theorem. First, we recall that $u\in W^{s,p}_0(\Omega)$ is a (weak) solution of problem \eqref{dir} iff for all $\varphi\in W^{s,p}_0(\Omega)$
\[\langle\fpl u,\varphi\rangle = \int_\Omega f(x,u)\varphi\,dx,\]
i.e., iff $J'(u)=0$ in $W^{-s,p'}(\Omega)$. We recall from \cite[Theorem 3.3, Remark 3.8]{CMS} the following {\em a priori} bound for weak solutions of \eqref{dir}:

\begin{lemma}\label{apb}
There exists $\eps_{0}=\eps_{0}(N, p, s, C_{0})>0$ such that if $u\in W^{s,p}_0(\Omega)$ solves \eqref{dir} under the growth condition \eqref{gro} and $K>0$ fulfills
\[\int_{\{|u|\ge K\}}|u|^{\p}\,dx<\eps_{0},\]
then $\|u\|_{\infty}\le C$ with $C=C(N, p, s, C_{0}, \|u\|_{s,p}, K)>0$.
\end{lemma}

\noindent
The bound in Lemma \ref{apb} is not uniform in $\|u\|_{s, p}$, due to the critical growth in \eqref{gro}, therefore in order to prove equi-boundedness of a sequence $(u_{n})_{n}$ of solutions to \eqref{dir} one not only needs an {\em a priori} bound on $\|u_{n}\|_{s, p}$, but also an equi-integrability estimate. For strictly subcritical reactions the dependance on $K$ can be dropped, and the former is sufficient.
\vskip2pt
\noindent
In addition to $C^0_s(\overline\Omega)$ defined in the Introduction, we will also use the following weighted H\"older space (see \cite{IMS} for details):
\[C^\alpha_s(\overline\Omega) = \Big\{u\in C^0(\overline\Omega):\,\frac{u}{\ds}\in C^\alpha(\overline\Omega)\Big\}, \qquad (\alpha\in(0,1)),\]
with norm
\[\|u\|_{C^{\alpha}_{s}}=\|u\|_{C^{0}_{s}}+\sup_{x\neq y}\frac{|u(x)/\ds(x)-u(y)/\ds(y)|}{|x-y|^\alpha}.\]
The embedding $C^\alpha_s(\overline\Omega)\hookrightarrow C^{0}_{s}(\overline\Omega)$ is compact for all $\alpha\in(0,1)$. The space $C^\alpha_s(\overline\Omega)$ is related to the global regularity theory for solutions of the following problem:
\beq\label{dirg}
\begin{cases}
\fpl u=g(x) & \text{in $\Omega$} \\
u=0 & \text{in $\R^N\setminus\Omega$,}
\end{cases}
\eeq
with $g\in L^\infty(\Omega)$. Weak solutions are defined just as those of \eqref{dir}. From \cite[Theorem 1.1]{IMS3} we have the following result:

\begin{lemma}\label{reg}
Let $p\ge 2$. Then, there exist $\alpha,C>0$, both depending on $\Omega$, $p$, and $s$, such that any weak solution $u\in W^{s, p}_0(\Omega)$ of \eqref{dirg} with $g\in L^\infty(\Omega)$ fulfills
\[\|u\|_{C^{\alpha}_{s}}\le C\,\|g\|_\infty^\frac{1}{p-1}.\]
\end{lemma}

\noindent
Another useful tool for our argument is the following monotonicity property of the fractional $p$-Laplacian, which we present separately in the degenerate and singular case.

\begin{lemma}\label{mon}
{\rm (degenerate case)} Let $p\ge 2$. There exists $C=C(p)>0$ such that for all $u,v\in W^{s, p}_{0}(\Omega)\cap L^\infty(\Omega)$ and all $q\ge 1$
\[\Big\|(u-v)^\frac{p+q-1}{p}\Big\|_{s, p}^p \le C\,q^{p-1}\,\langle\fpl u-\fpl v,(u-v)^q\rangle.\]
\end{lemma}
\begin{proof}
First we prove the following elementary inequality: for all $a,b,c,d\in\R$ such that $a-b=c-d$ it holds
\beq\label{abcd}
\Big|a^\frac{p+q-1}{p}-b^\frac{p+q-1}{p}\Big|^{p} \le C\,q^{p-1}\,\big(c^{p-1}-d^{p-1}\big)\,\big(a^q-b^q\big),
\eeq
with a constant $C=C(p)>0$ independent of $q$. We may assume that $a\ge b$ and $c\ge d$ since the former is equivalent to the latter, so that being $t\mapsto t^{r-1}$ increasing for all $r\ge 1$ all the factors of \eqref{abcd} are nonnegative. We apply \cite[Lemma A.2]{BP} with $g(t)=t^q$ and
\[G(t) = \int_0^t (g'(\tau))^\frac{1}{p}\, d\tau = \frac{p\,q^\frac{1}{p}}{p+q-1}t^\frac{p+q-1}{p}\]
to get
\[\Big(a^\frac{p+q-1}{p}-b^\frac{p+q-1}{p}\Big)^{p} \le \frac{(p+q-1)^p}{p^p\,q}\,\big(a-b\big)^{p-1}\,\big(a^q-b^q\big).\]
Besides, $p\ge 2$ implies that $t\mapsto t^{p-2}$ is increasing on $\R_{+}$, hence
\begin{align*}
c^{p-1}-d^{p-1} &= (p-1)\,\int_{d}^{c}|t|^{p-2}\,dt \\
&\ge (p-1)\,\int_{-(c-d)/2}^{(c-d)/2}|t|^{p-2}\,dt \\
&= \frac{1}{2^{p-2}}\,\big(c-d\big)^{p-1}.
\end{align*}
Recalling that $a-b=c-d$ and concatenating with the previous inequality, we get
\[\Big(a^\frac{p+q-1}{p}-b^\frac{p+q-1}{p}\Big)^{p} \le \frac{2^{p-2}(p+q-1)^p}{p^p\,q}\,\big(c^{p-1}-d^{p-1}\big)\,\big(a^q-b^q\big),\]
which yields \eqref{abcd}. Now pick $u,v\in W^{s,p}_{0}(\Omega)\cap L^\infty(\Omega)$, so $(u-v)^q\in W^{s, p}_0(\Omega)$. By \eqref{abcd} with $a=u(x)-v(x)$, $b=u(y)-v(y)$, $c=u(x)-u(y)$ and $d=v(x)-v(y)$, we have
\begin{align*}
\Big\|(u-v)^\frac{p+q-1}{p}\Big\|_{s,p}^p &= \iint_{\R^{2N}}\Big|(u(x)-v(x))^\frac{p+q-1}{p}-(u(y)-v(y))^\frac{p+q-1}{p}\Big|^p\,d\mu \\
&\le C\,q^{p-1}\,\langle\fpl u-\fpl v,(u-v)^q\rangle,
\end{align*}
which proves the assertion.
\end{proof}

\noindent
In the singular case, the monotonicity is slightly different:

\begin{lemma}\label{smon}
{\rm (singular case)} Let $p\in (1,2)$. Then, there exists $C=C(p)>0$ such that for all $u,v\in W^{s, p}_0(\Omega)\cap L^\infty(\Omega)$ and all $q\ge 1$
\[\frac{\big\|(u-v)^\frac{q+1}{2}\big\|_{s, p}^2}{\big(\|u\|_{s,p}^p+\|v\|_{s,p}^p\big)^{2-p}} \le C\, q \,\langle\fpl u-\fpl v,(u-v)^q\rangle.\]
\end{lemma}
\begin{proof}
Again we start with an elementary inequality:
\beq\label{sing}
\Big|a^\frac{q+1}{2}-b^\frac{q+1}{2}\Big|^2 \le C\, q\,\big(c^{p-1}-d^{p-1}\big)\,(a^q-b^q\big)\big(c^2+d^2\big)^\frac{2-p}{2},
\eeq
for all $a,b,c,d\in\R$ such that $a-b=c-d$, with a constant $C=C(p)>0$ independent of $q$. As in the previous proof, we may assume that $a\ge b$ and $c\ge d$. 
By the Cauchy-Schwartz inequality and the assumption $a-b=c-d$, we have
\begin{equation}\label{pla}
\begin{split}
\Big[a^\frac{q+1}{2}-b^\frac{q+1}{2}\Big]^2 &= \Big[\frac{q+1}{2}\,\int_{b}^{a}|t|^{\frac{q-1}{2}}\, dt\Big]^{2} \\
&\le \frac{(q+1)^{2}}{4}\,\int_{b}^{a}|t|^{q-1}\, dt\, (a-b) \\
&= \frac{(q+1)^{2}}{4 \, q}\,\big(a^{q}-b^{q}\big)\,(c-d).
\end{split}
\end{equation}
On the other hand, $|t|\le \big(c^2+d^2\big)^\frac{1}{2}$ for all $t\in [d, c]$, which, along with $p<2$, implies for all $t\in[d,c]$
\[|t|^{2-p}\le  \big(c^2+d^2\big)^\frac{2-p}{2}.\]
In turn, the latter implies
\begin{align*}
c-d &= \int_{d}^{c}|t|^{2-p}\,|t|^{p-2}\, dt \\
&\le \big(c^2+d^2\big)^\frac{2-p}{2}\,\int_{d}^{c}|t|^{p-2}\, dt \\
&= \frac{1}{p-1}\,\big(c^2+d^2\big)^\frac{2-p}{2}\,\big(c^{p-1}-d^{p-1}\big),
\end{align*}
which inserted into \eqref{pla} gives \eqref{sing}. Now pick $u,v\in W^{s, p}_0(\Omega)\cap L^\infty(\Omega)$ and set, for any $x, y\in \R^{N}$, $a=u(x)-v(x)$, $b=u(y)-v(y)$, $c=u(x)-u(y)$ and $d=v(x)-v(y)$. Using \eqref{sing} and H\"older's inequality with exponents $2/p$ and $2/(2-p)$, we get
\begin{align*}
&\big\|(u-v)^{\frac{q+1}{2}}\big\|_{s, p}^p = \iint_{\R^{2N}}\big|(u(x)-v(x))^\frac{q+1}{2}-(u(y)-v(y))^\frac{q+1}{2}\big|^p\,d\mu \\
&\le C\, q^{\frac{p}{2}}\iint_{\R^{2N}}\hspace{-2pt}\Big[\frac{\big((u(x)-u(y))^{p-1}\hspace{-2pt}-(v(x)-v(y))^{p-1}\big)\big((u(x)-v(x))^q-(u(y)-v(y))^q\big)}{\big((u(x)-u(y))^2+(v(x)-v(y))^2\big)^\frac{p-2}{2}}\Big]^\frac{p}{2}d\mu \\
&\le C\, q^{\frac{p}{2}}\,\big(\langle\fpl u-\fpl v,(u-v)^q\rangle\big)^\frac{p}{2}\big(\|u\|_{s,p}^p+\|v\|_{s,p}^p\big)^\frac{2-p}{2},
\end{align*}
with a different $C=C(p)>0$, still independent of $q$. Raising to the power $2/p$ we conclude.
\end{proof}

\section{Proof of the main result}\label{sec3}

\noindent
In this section we prove our main result:

\begin{proof}[Proof of Theorem \ref{svh}]
First we prove that \ref{svh1} implies \ref{svh2}. Assuming \ref{svh1}, we have in particular $J'(u_0)=0$ in $W^{-s, p'}(\Omega)$, hence by Lemma \ref{apb} $u_0\in L^\infty(\Omega)$. In turn, by \eqref{gro} we have $f(\cdot,u)\in L^\infty(\Omega)$. Then, Lemma \ref{reg} implies $u\in C^{0}_{s}(\overline\Omega)$.
\vskip2pt
\noindent
We argue by contradiction, assuming that there exists a sequence $(u_n)_{n}$ in $ W^{s, p}_0(\Omega)\cap C^{0}_{s}(\overline\Omega)$ such that $u_n\to u_0$ in $C^{0}_{s}(\overline\Omega)$ and $J(u_n)<J(u_0)$ for all $n\in\N$. Then we have $u_n\to u_0$ in $L^\infty(\Omega)$, hence
\[\lim_n\int_\Omega F(x,u_n)\,dx = \int_\Omega F(x,u_0)\,dx.\]
So we have
\begin{align*}
\limsup_n\frac{\|u_{n}\|_{s,p}^p}{p} &= \limsup_n \Big[J(u_n)+\int_\Omega F(x,u_n)\,dx\Big] \\
&\le J(u_0)+\int_\Omega F(x,u_0)\,dx = \frac{\|u_0\|_{s,p}^p}{p},
\end{align*}
in particular $(u_n)$ is bounded in $ W^{s, p}_0(\Omega)$. Passing to a subsequence, we have $u_n\rightharpoonup u_0$ in $ W^{s, p}_0(\Omega)$, hence
\[\|u_0\|_{s,p}\le\liminf_n\|u_n\|_{s,p}.\]
By the uniform convexity of $W^{s,p}_{0}(\Omega)$  it is easily seen that the latter implies $u_n\to u_0$ (strongly) in $ W^{s, p}_0(\Omega)$. Then, for all $n\in\N$ big enough we have $\|u_n-u_0\|_{s,p}\le\rho$ along with $J(u_n)<J(u_0)$, a contradiction. Thus, \ref{svh2} holds.
\vskip2pt
\noindent
Now we prove that \ref{svh2} implies \ref{svh1}. First note that, by \ref{svh2}, for all $\varphi\in W^{s, p}_0(\Omega)\cap C^{0}_{s}(\overline\Omega)$ we have
\[\langle J'(u_0),\varphi\rangle\ge 0.\]
Since $ W^{s, p}_0(\Omega)\cap C^{0}_{s}(\overline\Omega)$ is a dense subspace of $ W^{s, p}_0(\Omega)$, we have $J'(u_0)=0$ in $W^{-s, p'}(\Omega)$. As above, using Lemmas \ref{apb} and \ref{reg} we deduce that $u_0\in C^{0}_{s}(\overline\Omega)$, in particular $u_0\in L^\infty(\Omega)$. Again we argue by contradiction, assuming that there exists a sequence $(\tilde u_n)_{n}$ in $ W^{s, p}_0(\Omega)$ such that $\tilde u_n\to u_0$ in $ W^{s, p}_0(\Omega)$ and $J(\tilde u_n)<J(u_0)$ for all $n\in\N$. Set for all $n\in\N$
\[\delta_n:=\|\tilde u_n-u_0\|_{\p}, \qquad B_n=\big\{u\in W^{s, p}_0(\Omega):\,\|u-u_0\|_{\p}\le\delta_n\big\}.\]
By the continuous embedding $ W^{s, p}_0(\Omega)\hookrightarrow L^{\p}(\Omega)$ we have $\delta_n\to 0$, and $B_n$ is a closed convex (hence, weakly closed) subset of $ W^{s, p}_0(\Omega)$. Due to the critical growth in \eqref{gro}, we cannot directly minimize $J$ over $B_n$, so we introduce a suitable truncation. Set for all $t\in\R$, $\kappa>0$
\[[t]_\kappa={\rm sign}(t)\min\{|t|,\kappa\}.\]
For all $u\in W^{s, p}_0(\Omega)$ we have by dominated convergence
\beq\label{appr}
\lim_{\kappa\to +\infty}\int_\Omega\int_0^u f(x,[t]_\kappa)\,dt\,dx = \int_\Omega F(x,u)\,dx.
\eeq
Fix $n\in\N$, $\eps_{n}\in(0,J(u_0)-J(\tilde u_n))$. By \eqref{appr} we can find $\kappa_n>\|u_0\|_\infty+1$ such that
\[\Big|\int_\Omega F_n(x,\tilde u_n)\,dx-\int_\Omega F(x,\tilde u_n)\,dx\Big| < \eps_{n},\]
where we have set
\[f_n(x,t):=f(x,[t]_{\kappa_n}), \qquad F_n(x,t)=\int_0^t f_n(x,\tau)\,d\tau.\]
Note that, by \eqref{gro},
\[|f_n(x,t)| \le C_n:=C_{0}\,(1+\kappa_{n}^{\p-1}).\]
Set for all $u\in W^{s, p}_0(\Omega)$
\[J_n(u)=\frac{\|u\|_{s,p}^p}{p}-\int_\Omega F_n(x,u)\,dx,\]
hence $J_n\in C^1( W^{s, p}_0(\Omega))$, is sequentially weakly lower semi-continuous and, being $p>1$ and $|F_{n}(x, t)|\le C_{n}(1+|t|)$, it turns out to be coercive. Thus for any $n\ge 0$ we can find $u_{n}\in B_{n}$ solving the minimization problem
\beq\label{min}
J_n(u_n)=m_n=\inf_{u\in B_n}J_n(u).
\eeq
Notice that by the choice of $\eps_{n}$ and $\kappa_n$, we have
\beq
\label{jjn}
J_{n}(u_{n})\le J_n(\tilde u_n) \le J(\tilde u_n)+\eps_{n} < J(u_0) = J_n(u_0).
\eeq
We claim that there exists $\lambda_n\ge 0$ such that the following identity holds in $W^{-s, p'}(\Omega)$:
\beq\label{lag}
\fpl u_n+\lambda_n(u_n-u_0)^{\p-1} = f_n(x,u_n).
\eeq
Indeed, recalling that $u_n\in B_n$, two cases may occur:
\begin{itemize}
\item[$(a)$] If $\|u_n-u_0\|_{\p}<\delta_n$, then by \eqref{min} and the continuous embedding $ W^{s, p}_0(\Omega)\hookrightarrow L^{\p}(\Omega)$, $u_n$ is a local minimizer of $J_n$ in $ W^{s, p}_0(\Omega)$, hence $J'_n(u_n)=0$. So, \eqref{lag} holds with $\lambda_n=0$.
\item[$(b)$] If $\|u_n-u_0\|_{\p}=\delta_n$, then we apply Lagrange's multipliers rule. Set for all $u\in W^{s, p}_0(\Omega)$
\[I(u)=\frac{\|u-u_0\|_{\p}^{\p}}{\p}, \qquad \mathcal{M}_n=\Big\{u\in W^{s, p}_0(\Omega):\,I(u)=\frac{\delta_n^{\p}}{\p}\Big\},\]
then $I\in C^1( W^{s, p}_0(\Omega))$ and $\mathcal{M}_n$ is a $C^1$-manifold in $ W^{s, p}_0(\Omega)$. By \eqref{min}, $u_n$ is a global minimizer of $J_n$ on $\mathcal{M}_n$, so there exists $\lambda_n\in\R$ such that in $W^{-s, p'}(\Omega)$
\[J'_n(u_n)+\lambda_n I'(u_n) = 0,\]
which is equivalent to \eqref{lag}. Besides, by \eqref{min} again we have
\[\lambda_n = -\frac{\langle J'_n(u_n),u_0-u_n\rangle}{\langle I'(u_n),u_0-u_n\rangle} \ge 0.\]
\end{itemize}
By construction we have that $u_n\to u_0$ in $L^{\p}(\Omega)$, as $n\to\infty$. Moreover, by Lemma \ref{apb} and \eqref{lag}, we have $u_n\in L^\infty(\Omega)$. The next and most delicate step of the proof consists in proving that
\beq\label{uni}
u_n\to u_{0} \qquad \text{in $L^\infty(\Omega)$.}
\eeq
Adding \eqref{dir} and \eqref{lag}, for all $n\in\N$ we get in $W^{-s, p'}(\Omega)$
\beq\label{dif}
\fpl u_n-\fpl u_0+\lambda_n(u_n-u_0)^{\p-1} = g_n(x,u_n-u_0),
\eeq
where for all $(x,t)\in\Omega\times\R$ we have set
\[g_n(x,t) := f_n(x,t+u_0(x))-f(x,u_0(x)).\]
By \eqref{gro}, we can find $C>0$ (independent of $n$) such that for all $n\in\N$
\beq\label{difgro}
|g_n(x,t)| \le C\,(1+|t|^{\p-1}).
\eeq
We set $w_n=u_n-u_0\in W^{s, p}_0(\Omega)\cap L^\infty(\Omega)$ and test \eqref{dif} with $w_n^q\in W^{s, p}_0(\Omega)$, for $q\ge 1$:
\beq
\label{pqw}
\begin{split}
&\langle\fpl u_n-\fpl u_0,w_n^q\rangle+\lambda_n\int_\Omega|w_n|^{\p+q-1}\,dx = \int_\Omega g_n(x,w_n)\,w_n^q\,dx \\
&\le C\,\Big[\int_\Omega|w_n|^q\,dx+\int_\Omega|w_n|^{\p+q-1}\,dx\Big].
\end{split}
\eeq
We now apply Lemma \ref{mon} and the continuous embedding $ W^{s, p}_0(\Omega)\hookrightarrow L^{\p}(\Omega)$ to get
\begin{align*}
\Big[\int_\Omega |w_n|^\frac{\p(p+q-1)}{p}\,dx\Big]^\frac{p}{\p} &\le C\,\Big\|w_n^\frac{p+q-1}{p}\Big\|_{s,p}^p \\
&\le C\,q^{p-1}\,\langle\fpl u_n-\fpl u_0,w_n^q\rangle,
\end{align*}
which, along with $\lambda_n\ge 0$ and \eqref{pqw}, implies for all $n\in\N$, $q\ge 1$
\beq\label{qest}
\Big[\int_\Omega |w_n|^\frac{\p(p+q-1)}{p}\,dx\Big]^\frac{p}{\p} \le C\,q^{p-1}\,\Big[\int_\Omega|w_n|^q\,dx+\int_\Omega|w_n|^{\p+q-1}\,dx\Big],
\eeq
with $C>0$ independent of $q$ and $n$. Next, we shall derive an iterative formula from \eqref{qest}. We define recursively an increasing sequence $(q_j)_{j}$ by setting
\[q_1=1, \qquad q_{j+1}=\frac{\p (p+q_j-1)}{p},\]
so $q_j\to\infty$ as $j\to\infty$. In particular, we can find $\bar{\jmath}\in\N$ such that
\beq\label{qbar}
\bar q = q_{\bar{\jmath}} > \frac{N(\p-1)}{ps}.
\eeq
By \eqref{qest} with $q=q_j$, we have for all $j,n\in\N$
\[\Big[\int_\Omega |w_n|^{q_{j+1}}\,dx\Big]^\frac{p}{\p} \le C\,q_j^{p-1}\,\Big[\int_\Omega|w_n|^{q_j}\,dx+\int_\Omega|w_n|^{\p+q_j-1}\,dx\Big].\]
We aim at absorbing the last integral into the left hand side. By H\"older's inequality we have
\[\int_\Omega|w_n|^{\p+q_j-1}\,dx \le \Big[\int_\Omega|w_n|^{\p}\,dx\Big]^\frac{\p-p}{\p}\Big[\int_\Omega|w_n|^\frac{\p(p+q_j-1)}{p}\,dx\Big]^\frac{p}{\p}.\]
Since $w_n\to 0$ in $L^{\p}(\Omega)$, for all $n\in\N$ big enough we may assume
\[\Big[\int_\Omega|w_n|^{\p}\,dx\Big]^\frac{\p-p}{\p} \le \frac{1}{2\,C\,\bar q^{p-1}}.\]
So for all $j=1,\ldots\bar{\jmath}$ we have
\[\Big[\int_\Omega |w_n|^{q_{j+1}}\,dx\Big]^\frac{p}{\p} \le C\,q_j^{p-1}\,\int_\Omega|w_n|^{q_j}\,dx+\frac{1}{2}\,\Big[\int_\Omega|w_n|^{q_{j+1}}\,dx\Big]^\frac{p}{\p},\]
which yields
\beq\label{iter1}
\Big[\int_\Omega|w_n|^{q_{j+1}}\,dx\Big]^\frac{p}{\p} \le 2\,C\,q_j^{p-1}\,\int_\Omega|w_n|^{q_j}\,dx.
\eeq
Iterating on \eqref{iter1} for $j=1,\ldots\bar{\jmath}$, we find $\overline{C},\bar\alpha>0$ such that for all $n\in\N$ big enough
\[\|w_n\|_{\bar q} \le \overline{C}\,\|w_n\|_1^{\bar\alpha}.\]
Since $(w_n)_{n}$ is bounded in $L^1(\Omega)$, it is so in $L^{\bar q}(\Omega)$ as well. Now recall \eqref{qbar} and set
\[\bar r = \frac{\bar q}{\p-1} > \frac{N}{ps}.\]
By \eqref{difgro} and H\"older's inequality we have
\[\int_\Omega |g_n(x,w_n)|^{\bar r}\,dx \le C\int_\Omega(1+|w_n|^{\bar q})\,dx \le C.\]
 Again we test \eqref{dif} with $w_n^q\in W^{s, p}_0(\Omega)\cap L^{\infty}(\Omega)$, $q\ge 1$. As above, by using Lemma \ref{mon} we get
\begin{align}\label{est2}
\Big[\int_\Omega|w_n|^\frac{\p(p+q-1)}{p}\,dx\Big]^\frac{p}{\p} &\le C\,q^{p-1}\,\int_\Omega g_n(x,w_n)w_n^q\,dx \\
\nonumber &\le C\,q^{p-1}\,\Big[\int_\Omega |g_n(x,w_n)|^{\bar r}\,dx\Big]^\frac{1}{\bar r}\Big[\int_\Omega|w_n|^{q\bar r'}\,dx\Big]^\frac{1}{\bar r'} \\
\nonumber &\le C\,q^{p-1}\,\Big[\int_\Omega|w_n|^{q\bar r'}\,dx\Big]^\frac{1}{\bar r'},
\end{align}
with $C>1$ independent of $n$ and $q$. By \eqref{qbar} we may set
\[\gamma = \frac{\p}{p\bar r'} > 1,\]
and define recursively two sequences $(p_j)_{j}$, $(q_j)_{j}$ (different from the previous $(q_j)_j$) through
\[p_0 = \p, \qquad \ p_{j+1} = \gamma p_j+\frac{\p(p-1)}{p}, \qquad q_j = \frac{p_j}{\bar r'}.\]
So we have $p_j,q_j\to\infty$ as $j\to\infty$ and setting $q=q_j$ in \eqref{est2}, we have for all $n, j\in\N$
\[\Big[\int_\Omega |w_n|^{p_{j+1}}\,dx\Big]^\frac{p}{\p} \le C\,q_j^p\,\Big[\int_\Omega|w_n|^{p_j}\,dx\Big]^\frac{1}{\bar r'}\]
(recall that $q_j>1$, hence $q_j^{p-1}\le q_j^p$), which rephrases as the following recursive inequality:
\beq\label{rec}
\int_\Omega|w_n|^{p_{j+1}}\,dx \le (C\,q_j^p)^{\gamma \bar r'}\Big[\int_\Omega|w_n|^{p_j}\,dx\Big]^\gamma.
\eeq
Iterating on \eqref{rec} for $j\in\N$ and recalling that $q_j\sim\gamma^j/\bar r'$ as $j\to\infty$, we have
\begin{align*}
\int_\Omega|w_n|^{p_j}\,dx &\le \prod_{i=0}^{j-1}(C\,q_i^p)^{\gamma^{j-i}\bar r'}\,\Big[\int_\Omega|w_n|^{p_0}\,dx\Big]^{\gamma^j} \\
&\le C^{\gamma^j}\,\gamma^{p\bar r'\sum_{i=0}^{j-1} i\gamma^{j-i}}\Big[\int_\Omega|w_n|^{\p}\,dx\Big]^{\gamma^j},
\end{align*}
with an even bigger $C>1$ independent of $j$, $n$. Set
\[S=\sum_{i=0}^\infty i\,\gamma^{-i} < \infty,\]
then we have for all $n, j\in\N$
\[\int_\Omega|w_n|^{p_j}\,dx \le C^{\gamma^j}\,\gamma^{p\bar r'S\gamma^j}\,\Big[\int_\Omega|w_n|^{\p}\,dx\Big]^{\gamma^j}.\]
Note that, since $w_n\to 0$ in $L^{\p}(\Omega)$, the integral on the right hand side is less than $1$ for all $n\in\N$ big enough. Raising the last inequality to the power $1/p_j$, we get
\begin{align*}
\|w_n\|_{p_j} &\le C^\frac{\gamma^j}{p_j}\,\gamma^{p\bar r'S\frac{\gamma^j}{p_j}}\,\Big[\int_\Omega|w_n|^{\p}\,dx\Big]^\frac{\gamma^j}{p_j} \\
&\le C^\beta\,\gamma^{p\bar r'S\beta}\,\Big[\int_\Omega|w_n|^{\p}\,dx\Big]^\eta,
\end{align*}
where $\beta,\eta>0$ have been chosen such that for all $j\in\N$
\[\eta < \frac{\gamma^j}{p_j} < \beta.\]
Summarizing, we find $C>1$ such that for all $n, j\in\N$ big enough
\[
\|w_n\|_{p_j} \le C\,\|w_n\|_{\p}^{\eta\p}.
\]
Letting $j\to\infty$ and recalling that $w_n\in L^\infty(\Omega)$, we have for all $n\in\N$ large enough
\[\|w_n\|_\infty \le C\,\|w_n\|_{\p}^{\eta\p}.\]
Finally, from $w_n\to 0$ in $L^{\p}(\Omega)$ we infer $w_n\to 0$ in $L^\infty(\Omega)$ as well, thus proving \eqref{uni}.
\vskip2pt
\noindent
We can now conclude the proof. For $n\in\N$ big enough, \eqref{lag} rephrases as
\beq\label{lag1}
\fpl u_n = f(x,u_n)-\lambda_n(u_n-u_0)^{\p-1} \ \text{in $W^{-s, p'}(\Omega)$,}
\eeq
with $\lambda_n\ge 0$ (possibly $\lambda_n\to\infty$). As a consequence of \eqref{uni}, the sequence $(u_n)_{n}$ is bounded in $L^\infty(\Omega)$, so by \eqref{gro} we see that $(f(\cdot,u_n))_{n}$ is uniformly bounded as well. To go further we need a uniform bound on $(\lambda_n\,(u_n-u_0)^{\p-1})_{n}$. Testing again \eqref{dif} with $w_n^q$ (with $w_n=u_n-u_0$, $q\ge 1$) and applying Lemma \ref{mon}, we get for all $n\in\N$ big enough
\begin{align*}
\lambda_n\,\int_\Omega|w_n|^{\p+q-1}\,dx &\le \int_\Omega g(x,u_n)\, w_n^q\,dx \\
&\le C\,\int_\Omega |w_n|^q\,dx \\
&\le C\,\Big[\int_\Omega|w_n|^{\p+q-1}\,dx\Big]^\frac{q}{\p+q-1}|\Omega|^\frac{\p-1}{\p+q-1}
\end{align*}
(with $C>0$ independent of $n$, $q$), which implies
\[\lambda_n\,\|w_n\|_{\p+q-1}^{\p-1} \le C\,|\Omega|^\frac{\p-1}{\p+q-1}.\]
Letting $q\to\infty$, we have
\[\lambda_n\,\|w_n\|_\infty^{\p-1}\le C,\]
i.e., $(\lambda_n\,(u_n-u_0)^{\p-1})_{n}$ is a bounded sequence in $L^\infty(\Omega)$. Then, \eqref{lag1} and Lemma \ref{reg} imply that $(u_n)_{n}$ is bounded in $C^\alpha_s(\overline\Omega)$. By the compact embedding $C^\alpha_s(\overline\Omega)\hookrightarrow C_{s}^{0}(\bar \Omega)$, passing to a subsequence we have $u_n\to u_0$ in $C_{s}^{0}(\overline\Omega)$. So, for all $n\in\N$ big enough we have $\|u_n-u_0\|_{C^{0}_{s}}\le\sigma$.
\vskip2pt
\noindent
On the other hand, being $(u_{n})_{n}$ bounded in $L^{\infty}(\Omega)$, for $n$ large enough we have $J_{n}(u_{n})=J(u_{n})$, so that by \eqref{jjn} we have $J(u_n)<J(u_0)$. Thus we reached a contradiction to \ref{svh2}, and \ref{svh1} is proved.
\end{proof}

\section{An application}\label{sec4}

\noindent
To conclude we just want to give an example of how our result works, presenting a nonlinear extension of \cite[Theorem 3.3]{DI}. We make on the reaction $f$ in problem \eqref{dir} the following assumptions:
\begin{itemize}
\item[${\bf H}$] $f:\Omega\times\R\to\R$ is a Carath\'eodory map such that
\begin{enumroman}
\item\label{h1} $|f(x,t)|\le C_0(1+|t|^{q-1})$ for a.e.\ $x\in\Omega$, all $t\in\R$ ($C_0>0$, \ $q\in(p,\p)$);\vskip6pt
\item\label{h2} $f(x,t)t\ge 0$ for a.e.\ $x\in\Omega$, all $t\in\R$;\vskip3pt
\item\label{h3} $\displaystyle\limsup_{|t|\to\infty}\frac{F(x,t)}{|t|^p} \le 0$ uniformly for a.e.\ $x\in\Omega$;
\item\label{h4} $\displaystyle\liminf_{t\to 0}\frac{F(x,t)}{|t|^p} > \frac{\lambda_2}{p}$ uniformly for a.e.\ $x\in\Omega$.
\end{enumroman}
\end{itemize}
Here $\lambda_2>0$ denotes the second (variational) eigenvalue of $\fpl$ in $W^{s,p}_0(\Omega)$ (see \cite{BP} for details). Under these assumptions, we prove the following multiplicity result for problem \eqref{dir}:

\begin{theorem}\label{mult}
Let ${\bf H}$ be satisfied. Then, problem \eqref{dir} admits at least three nontrivial solutions.
\end{theorem}
\begin{proof}
We just sketch the proof, referring to \cite{DI} for details. First we introduce two truncated reactions and their primitives, defined for all $(x,t)\in\Omega\times\R$ by
\[f_\pm(x,t) = f(x,\pm t^\pm), \qquad F_\pm(x,t) = \int_0^t f_\pm(x,\tau)\,d\tau\]
(here $t^\pm=\max\{\pm t,0\}$), and note that $f_\pm$ are Carath\'eodory with subcritical growth due to ${\bf H}$ \ref{h1} \ref{h2}. We introduce the corresponding truncated energy functionals $J_\pm\in C^1(W^{s,p}_0(\Omega))$ defined by
\[J_\pm(u) = \frac{\|u\|_{s,p}^p}{p}-\int_\Omega F_\pm(x,u)\,dx.\]
By ${\bf H}$ \ref{h3}, $J_+$ is coercive and thus it admits a global minimizer $u_+\in W^{s,p}_0(\Omega)$. Using ${\bf H}$ \ref{h4}, we easily see that $u_+\neq 0$. Clearly, $u_+$ is a weak solution of the auxiliary problem
\[\begin{cases}
\fpl u = f_+(x,u) & \text{in $\Omega$} \\
u = 0 & \text{in $\R^N\setminus\Omega$,}
\end{cases}\]
hence by Lemmas \ref{apb}, \ref{reg} we have $u_+\in C^0_s(\overline\Omega)$. Besides, by the fractional $p$-Laplacian Hopf's lemma (see \cite[Theorem 1.5]{DQ}), we have uniformly for all $x\in\partial\Omega$
\[\lim_{\Omega\ni y\to x}\frac{u_+(y)}{\ds(y)} > 0.\]
Thus, $u_+$ lies in the interior of the positive order cone $C^0_s(\overline\Omega)_+$ of $C^0_s(\overline\Omega)$ (see \cite[Lemma 5.1]{ILPS}). Since $J_+=J$ on $C^0_s(\overline\Omega)_+$, we see that $u_+$ is a local minimizer of $J$ in the $C^0_s(\overline\Omega)$-topology. By Theorem \ref{svh}, $u_+$ is as well a local minimizer of $J$ in the $W^{s,p}_0(\Omega)$-topology.
\vskip2pt
\noindent
Arguing similarly on $J_-$, we detect a local minimizer $u_-\in -{\rm int}(C^0_s(\overline\Omega)_+)$ of $J$. Plus, $J$ satisfies the Palais-Smale condition, hence by the mountain pass theorem it admits one more critical point $\tilde u\in W^{s,p}_0(\Omega)$. Exploiting condition ${\bf H}$ \ref{h4} and the variational characterization of the second eigenvalue $\lambda_2$ (see \cite[Theorem 5.3]{BP}), we see that $\tilde u\neq 0$. Finally, we use Lemmas \ref{apb}, \ref{reg} to deduce that $\tilde u\in C^0_s(\overline\Omega)$.
\vskip2pt
\noindent
All in all, $u_+,u_-,\tilde u\in W^{s,p}_0(\Omega)\cap C^0_s(\overline\Omega)\setminus\{0\}$ are three nontrivial solutions of \eqref{dir}.
\end{proof}

\vskip10pt
\noindent
{\small {\bf Acknowledgement.} All authors are members of GNAMPA (Gruppo Nazionale per l'Analisi Matematica, la Probabilit\`a e le loro Applicazioni) of INdAM (Istituto Nazionale di Alta Matematica 'Francesco Severi'). A.\ Iannizzotto and S.\ Mosconi are supported by the grant PRIN n.\ 2017AYM8XW: {\em Non-linear Differential Problems via Variational, Topological and Set-valued Methods}. A.\ Iannizzotto is also supported by the research project {\em Integro-differential Equations and nonlocal Problems}, funded by Fondazione di Sardegna (2017), S.\ Mosconi by the grant PdR 2018-2020 - linea di intervento 2: {\em Metodi Variazionali ed Equazioni Differenziali} of the University of Catania.}

\end{document}